\documentclass[11pt]{amsart}

\usepackage{amsmath}
\usepackage{amssymb}
\usepackage{graphicx}

\hoffset = - 0.5in

\newtheorem{theorem}{Theorem}[section]

\newtheorem{proposition}[theorem]{Proposition}
\theoremstyle{definition}
\newtheorem{definition}[theorem]{Definition}

\numberwithin{equation}{section}

\renewcommand{\geq}{\geqslant}
\renewcommand{\leq}{\leqslant}

\title{$X^*$ with weak* uniform Kadec-Klee property has Property($K^*$)}

\author[Tim Dalby]{Tim Dalby}

\date{\today}



\email{tim\_dalby@bigpond.com}

\begin{document}

\parindent = 0pt
\parskip = 8pt

\begin{abstract}

It is shown that if the dual of a Banach space, $X^*$, where the dual ball is weak* sequentially compact, has the weak* uniform Kadec-Klee property then $X^*$ has Property($K^*$).  An example is given where the reverse implication does not hold.  That is, there is a Banach space $X$ whose dual, $X^*$, has Property($K^*$) but $X^*$ does not have the weak* uniform Kadec-Klee property.

\end{abstract}

\maketitle

\section{Introduction}

A Banach space, $X$, has the weak fixed point property (w-FPP) if every nonexpansive mapping, $T$, on every weak compact convex nonempty subset, $C$, has a fixed point.  If the subsets are closed and bounded instead of weak compact then the property is called the fixed point property (FPP).  The past 40 or so years has seen a number of Banach space properties shown to imply the w-FPP.  Some such properties are weak normal structure, Opial's condition, Property($K$)  and Property($M$).  There has also been interest in how these properties are linked or interact. See for example [1], [4] and [9].

In [6] Dowling, Randrianantoanina and Turett showed that if $X$ is a Banach space with the dual unit ball, $B_{X^*}$, weak* sequentially compact and $X^*$ has the weak* uniform Kadec-Klee property then $X$ has the w-FPP.   In a \textquoteleft Note added in proof \textquoteright\, the authors thanked Garc\'ia-Falset for pointing out that if $X^*$ has the weak* uniform Kadec-Klee  propery then $R(X) < 2.$  This last property, in turn, implies $X$ has the w-FPP.  The proof of that last result appears in [7].

\bigskip

On the other hand, Dalby in [2] showed that $X^*$ having Property($K^*$) implies that $R(X) < 2.$   So could there be a link between these two properties in $X^*?$  This paper produces that link; if $X^*$ has the weak* uniform Kadec-Klee property then $X^*$ has Property($K^*$).

Definitions for the properties mentioned above appear in the next section.  Note that because the w-FPP is separably determined, the Banach spaces are assumed to be separable.

\bigskip

To finish things off, an example is produced where $X^*$ has Property($K^*$) but $X^*$ does not have the weak* uniform Kadec-Klee property.

\section {Definitions}

Below are the definitions of properties that will be used in the next section.

\begin{definition}

Sims, [12]

A Banach space $X$ has property($K$) if there exists $K \in [0, 1)$ such that whenever $x_n \rightharpoonup 0, \lim_{ n \rightarrow \infty} \| x_n \| = 1 \mbox{ and }  \liminf_{n \rightarrow \infty} \| x_n - x \|  \leq 1 \mbox{ then } \| x \| \leq K.$

If the sequence is in $B_{X^*}$ and is weak* convergent to zero then the property is called Property($K^*$).

\end{definition}

\begin{definition}

Lin, Tan and Xu, [10]

Opial's modulus is
\[ r_X(c) := \inf \{ \liminf_{n\rightarrow \infty} \| x_n - x \| - 1: c \geq 0, \| x \| \geq c,  x_n \rightharpoonup 0 \mbox{ and }\liminf_{n\rightarrow \infty} \| x_n \| \geq 1 \}. \]

Note that $X$ has uniform Opial's condition if $r_X(c) > 0$ for all $c > 0.$  See [10] or [9] for more details.

\end{definition}

\begin{definition}

A dual of a Banach space has the weak* uniform Kadec-Klee property if for every $\epsilon > 0$ there exist a $\delta > 0$ such that if $(x_n^*)$ is in $B_{X^*}$ and converges weak* to $x^*$ and the separation constant, 

\[ \mbox{ sep } (x_n^*) := \inf  \{ \| x_n^* - x_m^* \|: m \ne n \} > \epsilon \]

then $\| x_n^* \| < 1 - \delta.$

\end{definition}

\section{Result}

In [1] Dalby showed that if $X$ is a separable Banach space then $X$ has Property($K$) if and only if $r_X(1) > 0.$  The proof readily transfers to the dual so Property($K^*$) is equivalent to $r_{X^*}(1) > 0.$

\begin{proposition}   Let $X$ be a Banach space with $B_{X^*}$ weak* sequentially compact and where $X^*$ has the weak* uniform Kadec-Klee property then $r_{X^*}(1) > 0$.
\end{proposition}

\begin{proof}  Assume $r_{X^*}(1) \leq 0$ then $r_{X^*}(1) = 0$.  Given $0 < \epsilon < 1$ there exists $(x_n^*)$ and $x^*$ in $X$ where $x_n^*\stackrel{*} \rightharpoonup 0, \liminf_{n \rightarrow\infty} \| x_n^* \| \geq 1, \| x^* \| \geq 1$ and

\[ \liminf_{n \rightarrow \infty} \| x_n^* + x^* \| - 1 < r_{X^*}(1) + \epsilon = \epsilon. \] 

Then $\liminf_{n \rightarrow\infty} \| x_n^* + x^* \| < 1 + \epsilon$ and by taking subsequences, if necessary, we may assume that

\[ \frac{x_n^* + x^*}{1 + \epsilon}\leq 1 \mbox{ for all } n. \] 

Again by taking subsequences, we may assume that sep$(x_n^*) > \epsilon.$  Then

\[ \mbox{sep}\left ( \frac{x_n^*+x^*}{1 + \epsilon}\right ) = \mbox{sep} \left (\frac{x_n^*}{1 + \epsilon}\right ) > \frac{\epsilon}{1 + \epsilon} > \frac{\epsilon}{2}. \] 

Therefore, by the weak* uniform Kadec-Klee property of $X^*$, there exists $\delta > 0$ such that

\[ \left \| \frac{x^*}{1 + \epsilon}\right \| < 1 - \delta. \]

For $\epsilon < \delta, \, 1 \leq \| x^* \| < (1 - \delta)(1 + \epsilon) < (1 - \delta)(1 + \delta) = 1 - \delta^2$, which is a contradiction.  Therefore $r_{X^*}(1) > 0$.

\end{proof}

{\bf Example}

Let $X = (l_2 \oplus l_3 \oplus \ldots \oplus l_n \oplus \ldots)_2$ then $X$ is a reflexive Banach space that does not have the weak* uniform Kadec-Klee property.  For more details see Sims [11].  But $X$ has the uniform Opial condition, as shown below. So, in particular, $r_X(1) > 0.$

The proof that $X$ has the uniform Opial condition comes from Property(D) of Dalby and Sims [3].  A Banach space has Property(D) if there exists an increasing strictly positive function $\alpha$ on $(0, \infty)$ such that whenever

$x_n \rightharpoonup x_\infty \ne 0, \lim_{n \rightarrow\infty}\| x_n - x_\infty \| = 1, \mbox{ and } x_n^* \in J(x_n),$ we have

\[ \liminf_{n \rightarrow \infty} x_n^*(x_\infty) \geq \alpha(\| x_\infty \|). \] 

Here $J:X \rightarrow X^*$ is the duality map,

\[ J(x) = \{ x^* \in X^*: x^*(x) = \| x \|^2 \mbox{ and } \| x^* \| = \| x \|\}. \] 

Dalby and Sims showed that a Banach space, $X$, has the uniform Opial condition if and only if $X$ has property(D).

Consider $(x_n)$ in $X$ where $x_n \rightharpoonup x_\infty \ne 0 \mbox{ and } \lim_{n \rightarrow\infty}\| x_n - x_\infty \| = 1.$

Let $x_n = (x_n^p) \mbox{ and } x_\infty = (x_\infty^p) \mbox{ where } x_n^p,x_\infty^p  \in l_p .$  Then, because $l_p$ is smooth, there is a unique element, $x_n^{p,*} \mbox{ of } J(x_n^p).$  It follows that $x_n^* = (x_n^{p,*})$ is the unique element of $J(x_n).$

For each $p$, if $y_n \rightharpoonup y_\infty \ne 0 \mbox{ and } y_n^* \in J(y_n) \mbox{ then } \| y_n \|_p^{p - 2}y_n^*\rightharpoonup \| y_\infty |_p^{p-2}y_\infty^*.$  For more details see [11].

Thus $\liminf_{n \rightarrow\infty} y_n^*(y_\infty) \geq \| y_\infty \|_p^2.$ 

Let $N = \{p \in \mathbb{N}: x_\infty^p \ne 0\}$ then $\| x_\infty \| = \left ( \sum_{p \in N}\| x_\infty^p \|^2 \right )^\frac{1}{2}.$  Hence

\begin{align*}
\liminf_{n \rightarrow \infty} x_n^*(x_\infty) & = \liminf_{n \rightarrow \infty}\sum x_n^{p,*}(x_\infty^p) \\
& = \sum_{p \in N} \liminf_{n \rightarrow \infty} x_n^{p,*}(x_\infty^p) \\
& \geq \sum_{p \in N} \| x_\infty \|^2 \\
& = \| x_\infty \|^2.
\end{align*}

The required increasing strictly positive function is $\alpha(x) = x^2$ and so $X$ has the uniform Opial condition which ensures $r_X(1) > 0.$

\end{document}